\documentclass[12pt,english,reqno]{amsart}

\usepackage{amsthm,amsmath,amssymb,amsfonts}
\usepackage{mathtools}
\usepackage{mleftright}

\usepackage[
  pdfauthor={},
  pdfkeywords={},
  pdftitle={},
  pdfcreator={},
  pdfproducer={},
  linktocpage,colorlinks,bookmarksnumbered]{hyperref}

\theoremstyle{plain}
\newtheorem{theorem}{Theorem}
\newtheorem{lemma}{Lemma}
\newtheorem{corollary}{Corollary}

\theoremstyle{definition}

\newtheorem*{remark}{Remark}

\newcommand{\NN}{\mathbb{N}}
\newcommand{\ZZ}{\mathbb{Z}}
\newcommand{\QQ}{\mathbb{Q}}
\newcommand{\KK}{\mathcal{K}}
\newcommand{\PP}{\mathcal{P}}
\newcommand{\BT}{\widetilde{B}}

\newcommand{\set}[1]{\left\{#1\right\}}
\newcommand{\intpart}[1]{\mleft[#1\mright]}
\newcommand{\fracpart}[1]{\mleft\{#1\mright\}}
\newcommand{\fracsym}[2]{\mleft\{\!\!\!\mleft\{\frac{#1}{#2}\mright\}\!\!\!\mright\}}

\DeclareMathOperator{\denom}{denom}
\DeclareMathOperator{\ord}{ord}

\newcommand{\arXiv}[1]{\href{http://arxiv.org/abs/#1}{\texttt{arXiv:\,#1~[math.NT]}}}

\begin{document}

\title[Distribution modulo one and denominators]{Distribution modulo one\\
and denominators of\\ the Bernoulli polynomials}
\author{Bernd C. Kellner}
\address{G\"ottingen, Germany}
\email{bk@bernoulli.org}

\begin{abstract}
Let $\{\cdot\}$ denote the fractional part and $n \geq 1$ be a fixed integer.
In this short note, we show for any prime $p$ the one-to-one correspondence
$$
  \sum_{\nu \geq 1} \left\{\frac{n}{p^\nu}\right\} > 1
  \quad \iff \quad
  p \mid \mathrm{denom}( B_n(x) - B_n ),
$$
where $B_n(x) - B_n$ is the $n$th Bernoulli polynomial without constant term
and $\mathrm{denom}(\cdot)$ is its denominator, which is squarefree.
\end{abstract}

\maketitle


\section{Introduction}

Recently, the properties of the denominators of the Bernoulli polynomials
$B_n(x)$ and $B_n(x) - B_n$ have sparked interest of several authors, see
\cite{BLMS:2017,Kellner:2017,Kellner&Sondow:2017,Kellner&Sondow:2017b}.
Various related sequences of these denominators were studied in
\cite{Kellner&Sondow:2017,Kellner&Sondow:2017b}.

The Bernoulli polynomials are defined by
\begin{align}
  \frac{t e^{xt}}{e^t - 1} &= \sum_{n \geq 0} B_n(x) \frac{t^n}{n!}
    \quad (|t| < 2\pi), \nonumber \\
\shortintertext{where}
  B_n(x) &= \sum_{k=0}^{n} \binom{n}{k} B_k \, x^{n-k} \quad (n \geq 0)
    \label{eq:bn-poly}
\end{align}
and $B_k = B_k(0)$ is the $k$th Bernoulli number
(cf.~\cite[Chap.~3.5, pp.~112]{Prasolov:2010}).

Let $p$ denote always a prime. The denominators of the Bernoulli polynomials can
be explicitly characterized as follows, which involves the function $s_p(n)$
giving the sum of base-$p$ digits of~$n$.

\begin{theorem}[Kellner and Sondow \cite{Kellner&Sondow:2017}] \label{thm:denom}
For $n \geq 1$, we have
\[
  \denom \bigl( B_n(x) - B_n \bigr)
  = \prod_{ \substack{p \, \leq \, \frac{n+1}{\lambda_n} \\
    s_p(n) \, \geq \, p}} p
\]
with
\[
  \lambda_n :=
    \begin{cases}
      \, 2, & \text{if $n$ is odd,} \\
      \, 3, & \text{if $n$ is even.}
    \end{cases}
\]
The bound $\frac{n+1}{\lambda_n}$ is sharp for odd $n=2p-1$ and even $n=3p-1$,
when $p$ is an odd prime, respectively.
\end{theorem}

While the sharp bounds $\frac{n+1}{\lambda_n}$ in Theorem~\ref{thm:denom} give
a stronger result, the author \cite{Kellner:2017} has subsequently shown that
they can be omitted yielding a more suitable formula:
\begin{equation} \label{eq:denom-poly}
  \denom \bigl( B_n(x) - B_n \bigr) = \prod_{s_p(n) \, \geq \, p} p
  \quad (n \geq 1).
\end{equation}

In this short note, we want to focus on the following theorem relating
\eqref{eq:denom-poly} to sums of fractional parts;
the latter denoted by $\fracpart{\,\cdot\,}$.

\begin{theorem} \label{thm:main}
If $n \geq 1$, then for any prime $p$ we have
\begin{equation} \label{eq:frac-denom}
  \sum_{\nu \geq 1} \fracpart{\frac{n}{p^\nu}} > 1
  \quad \iff \quad
  p \mid \denom \bigl( B_n(x) - B_n \bigr).
\end{equation}
This is a one-to-one correspondence, since the denominator of the Bernoulli
polynomial $B_n(x) - B_n$ is squarefree.
\end{theorem}

At first glance, this remarkable connection between the distribution modulo one
and the denominators of the Bernoulli polynomials seems to be ``mysterious''.
It remains an open question, whether there exists a more general law that is
hidden behind this relation.

Clearly, the sequence $\left(\fracpart{n p^{-\nu}}\right)_{\nu \geq 1}$
eventually converges to zero for any $n \geq 1$ and prime $p$. We have
a partition into a geometric series and a finite sum of fractional parts,
which turns into
\begin{equation} \label{eq:frac-finite}
  \sum_{\nu \geq 1} \fracpart{\frac{n}{p^\nu}}
  = \frac{n}{p^\ell (p-1)} + \sum_{\nu = 1}^{\ell} \fracpart{\frac{n}{p^\nu}},
\end{equation}
where $\ell \geq 0$ satisfying $p^\ell \leq n < p^{\ell+1}$.
For $p > n$ the right-hand side of \eqref{eq:frac-finite} reduces to
$n/(p-1) \leq 1$, implying that \eqref{eq:frac-denom} only holds for finitely
many primes $p$. Similarly, we can derive from \eqref{eq:frac-finite} the same
bounds as in Theorem~\ref{thm:denom}.

\begin{lemma}
If $n \geq 1$ and $p$ is a prime, then
\[
  p > \frac{n+1}{\lambda_n}
  \quad \implies \quad
  \sum_{\nu \geq 1} \fracpart{\frac{n}{p^\nu}} \leq 1,
\]
where $\lambda_n$ is defined as in Theorem~\ref{thm:denom}.
\end{lemma}

\begin{proof}
As pointed out above, for $p > n$ the claim already holds.
We have to distinguish between two cases as follows.

Case $p > \frac{n+1}{2}$:
We obtain the bounds $2p-2 \geq n \geq p$. Thus $n = p+r$ with $0 \leq r \leq p-2$.
The right-hand side of \eqref{eq:frac-finite} yields with $a=1$ that
\begin{equation} \label{eq:frac-estim}
  \frac{n}{p(p-1)} + \fracpart{\frac{n}{p}}
  = \frac{ap+r}{p(p-1)} + \frac{r}{p} = \frac{r+a}{p-1} \leq 1,
\end{equation}
showing the result for the first case.

Case $p > \frac{n+1}{3}$ and $n$ even:
With the first case there remain the bounds $3p-2 \geq n \geq 2p$.
If $p=2$, then only $n = p^2$ can hold. By definition,
the left-hand side of \eqref{eq:frac-finite} evaluates to $1/(p-1) = 1$.
For odd $p \geq 3$ we infer that $n = 2p + r$ with $0 \leq r \leq p-3$.
The right-hand side of \eqref{eq:frac-finite} then becomes \eqref{eq:frac-estim}
with $a=2$. This completes the second case and shows the result.
\end{proof}

On the other side, the next theorem shows that the values of the sum of
fractions can be arbitrarily large for powers of $n$.

\begin{theorem} \label{thm:power}
Let $n > 1$ and $p$ be a prime. If $n$ is not a power of $p$, then
\[
  \sum_{\nu \geq 1} \fracpart{\frac{n^k}{p^\nu}} \to \infty
  \quad \text{as} \quad k \to \infty.
\]
\end{theorem}

In view of Theorem~\ref{thm:main}, a weaker version of Theorem~\ref{thm:power},
that the above sum is greater than $1$ for all sufficiently large values of $k$,
already implies the following corollary.

\begin{corollary} \label{cor:power}
Let $\PP$ be a finite set of primes and
\[
  \Pi := \prod_{p \in \PP} p.
\]
If $n > 1$ is not a power of any $p \in \PP$,
then there exists a constant $M$ depending on $n$ and $\PP$ such that
\[
  \Pi \mid \denom \bigl( B_{n^k}(x) - B_{n^k} \bigr) \quad (k \geq M).
\]
\end{corollary}

The aim of this paper is to give a somewhat elementary and direct proof of
Theorem~\ref{thm:main} using properties of fractional parts. This results in
new variants of proofs given in the next sections.


\section{Preliminaries}

Let $\ZZ_p$ be the ring of $p$-adic integers, and $\QQ_p$ be the field of $p$-adic
numbers. Define $\ord_p s$ as the $p$-adic valuation of $s \in \QQ_p$. Let
$\intpart{\,\cdot\,}$ denote the integer part. For $n \geq 0$ define the symbol
\begin{align}
  \fracsym{n}{p} :=& \sum_{\nu \geq 1} \fracpart{\frac{n}{p^\nu}}. \nonumber \\
\intertext{By Legendre's formula (see \cite[Chap.~5.3, p.~241]{Robert:2000}) we have}
  \ord_p n! =& \sum_{\nu \geq 1} \intpart{\frac{n}{p^\nu}}. \label{eq:fac-val}
\end{align}
Writing $n/(p-1)$ as a geometric series and using 
$\fracpart{x} = x - \intpart{x}$ yield
\begin{equation} \label{eq:fs-eq}
  \fracsym{n}{p} = \frac{n}{p-1} - \ord_p n!.
\end{equation}
Consequently, we have
\begin{equation} \label{eq:fs-rel}
  \fracsym{n}{p} \in \NN \quad \iff \quad p - 1 \mid n.
\end{equation}

\begin{lemma} \label{lem:fs-decomp}
Let $p$ be a prime. If $a \geq 0$ and $0 \leq r < p$, then
\[
  \fracsym{ap+r}{p} = \fracsym{a}{p} + \fracsym{r}{p}.
\]
\end{lemma}

\begin{proof}
Since $0 \leq r < p$, we observe by \eqref{eq:fac-val} that
\[
  \ord_p \, (ap+r)! = \sum_{\nu \geq 1} \intpart{\frac{ap+r}{p^\nu}}
  = a + \ord_p a!.
\]
Hence, the result follows easily by \eqref{eq:fs-eq}.
\end{proof}

\begin{lemma} \label{lem:fs-decomp2}
If $n \geq 1$ and $p$ is a prime, then
\[
  \fracsym{n}{p} = \sum_{j=0}^{\ell} \fracsym{n_j}{p},
\]
where $n = n_0 + n_1 \, p + \dotsb + n_\ell \, p^\ell$ is
the $p$-adic expansion of $n$.
\end{lemma}

\begin{proof}
This follows by applying Lemma~\ref{lem:fs-decomp} iteratively.
\end{proof}

\begin{lemma} \label{lem:fs-val-binom}
If $n > 1$ and $p \leq n$ is a prime, then
\[
  \fracsym{n}{p} \leq 1 \quad \iff \quad \ord_p \binom{n}{k} \geq 1
\]
for all $0 < k < n$, where $p-1 \mid k$.
\end{lemma}

\begin{proof}
Set $\KK := \set{k \in \NN : 0 < k < n, \, p-1 \mid k}$, where
$\KK \neq \emptyset$ by assumption. We consider the following two cases.

Case $\fracsym{n}{p} \leq 1$:
Applying \eqref{eq:fs-eq} and \eqref{eq:fs-rel} we infer that
\begin{equation} \label{eq:binom-val}
  \ord_p \binom{n}{k} = \underbrace{\fracsym{n-k}{p}}_{> \, 0}
    + \underbrace{\fracsym{k}{p}}_{\geq \, 1}
    - \underbrace{\fracsym{n}{p}}_{\leq \, 1} > 0
    \quad (k \in \KK).
\end{equation}

Case $\fracsym{n}{p} > 1$:
We will show that an integer $k \in \KK$ exists with $\ord_p \binom{n}{k} = 0$.
Generally, there exist integers $k$ with $1 \leq k \leq n$
such that the $p$-adic expansions yield
\begin{alignat}{5}
  n &= n_0 &&+ n_1 \, p &&+ \dotsb + n_\ell \, p^\ell, \nonumber \\
  k &= k_0 &&+ k_1 \, p &&+ \dotsb + k_\ell \, p^\ell \label{eq:exp-k}
\end{alignat}
with $0 \leq k_j \leq n_j$ for $j = 0,\dotsc,\ell$. Thus, we have the
$p$-adic expansion
\begin{equation}
  n-k = (n_0-k_0) + (n_1-k_1) \, p + \dotsb + (n_\ell - k_\ell) \, p^\ell.
\end{equation}
With that we achieve by Lemma~\ref{lem:fs-decomp2} and using \eqref{eq:fs-eq} that
\begin{equation} \label{eq:fs-n-k}
  \fracsym{n}{p} = \fracsym{n-k}{p} + \fracsym{k}{p},
\end{equation}
implying by \eqref{eq:binom-val} that
\begin{equation} \label{eq:binom-val-0}
  \ord_p \binom{n}{k} = 0.
\end{equation}

Since $\fracsym{n}{p} > 1$, we obtain by Lemma~\ref{lem:fs-decomp2} and
\eqref{eq:fs-eq} that
\begin{alignat*}{5}
  n_0 &+ n_1 &&+ \dotsb + n_\ell &&> p-1.
\intertext{Hence, we can choose $k_j$ with $0 \leq k_j \leq n_j$ for
$j = 0,\dotsc,\ell$ such that}
  k_0 &+ k_1 &&+ \dotsb + k_\ell &&= p-1,
\end{alignat*}
satisfying \eqref{eq:exp-k} -- \eqref{eq:binom-val-0} but with $0 < k < n$.
By Lemma~\ref{lem:fs-decomp2} and \eqref{eq:fs-eq} it then follows that
$\fracsym{k}{p} = 1$.
Consequently, $p-1 \mid k$ by \eqref{eq:fs-rel} and therefore $k \in \KK$.
This completes the second case.

Finally, both cases imply the claimed equivalence.
\end{proof}

\begin{remark}
Actually, Eqs.~\eqref{eq:binom-val}, \eqref{eq:fs-n-k}, and \eqref{eq:binom-val-0}
reflect Kummer's theorem that $\ord_p \binom{n}{k}$ equals the number of carries
when adding $k$ to $n-k$ in base~$p$. Carlitz \cite{Carlitz:1961} gave a more
general result of Lemma~\ref{lem:fs-val-binom} in context of the function~$s_p(n)$,
whose proof depends on Lucas's theorem, namely
\[
  \binom{n}{k} \equiv \binom{n_0}{k_0} \binom{n_1}{k_1} \dotsm
    \binom{n_\ell}{k_\ell} \pmod{p}.
\]
\end{remark}


\section{Proof of \texorpdfstring{Theorem~\ref{thm:main}}{Theorem~2}}

While the Bernoulli numbers $B_n = 0$ for odd $n \geq 3$,
the theorem of von Staudt--Clausen asserts for even $n \geq 2$ that
\begin{equation} \label{eq:denom-bn}
  B_n + \sum_{p-1 \, \mid \, n} \frac{1}{p} \in \ZZ
  \quad \text{implying that} \quad
  \denom( B_n ) = \prod_{p-1 \, \mid \, n} p.
\end{equation}

The $p$-adic valuation of a nonzero polynomial
\[
  f(x) = \sum_{k=0}^{r} a_k \, x^k \in \QQ[x] \backslash \! \set{0}
\]
of degree $r$ is given by
\[
  \ord_p f(x) = \min_{ 0 \leq k \leq r} \ord_p a_k.
\]

Define the polynomials
\begin{align*}
  \BT_n(x) &:= B_n(x) - B_n, \\
  \BT_{n,p}(x) &:= \!
    \sum_{\substack{k=2\\ 2 \, \mid \, k \\ p-1 \, \mid \, k}}^{n-1}
    \!\! \binom{n}{k} B_k \, x^{n-k}.
\end{align*}

\begin{lemma} \label{lem:bt-p-val}
If $n \geq 3$ and $p$ is a prime, then
\[
  \ord_p \BT_{n,p}(x) =
    \left\{
      \begin{array}{rl}
        -1, & \text{if } \fracsym{n}{p} > 1, \\
        \geq 0, & \text{else.}
      \end{array}
    \right.
\]
In particular, if $p=2$ and $n \geq 3$ is odd, then $\ord_p \BT_{n,p}(x) = -1$.
\end{lemma}

\begin{proof}
Set $\KK_2 := \set{k \in 2\NN : 0 < k < n, \, p-1 \mid k}$.
Note that $\KK_2 = \emptyset \Leftrightarrow p > n$. In this case,
we have $\BT_{n,p}(x) = 0$ and thus $\ord_p \BT_{n,p}(x) = \infty$,
as well as $\fracsym{n}{p} \leq 1$ by \eqref{eq:fs-eq}.

So we assume that $\KK_2 \neq \emptyset$ and $p \leq n$.
For all coefficients of $\BT_{n,p}(x)$, we obtain by \eqref{eq:denom-bn} that
\begin{equation} \label{eq:coeff-val}
  \ord_p \binom{n}{k} B_k = -1 + \ord_p \binom{n}{k} \geq -1
    \quad (k \in \KK_2).
\end{equation}
For odd primes $p$ the claim follows at once by Lemma~\ref{lem:fs-val-binom}.
For $p = 2$ there remain two cases as follows.

Case $p=2$ and $n \geq 3$ odd:
As $n-1 \in \KK_2$ and $n = \binom{n}{n-1}$ is odd, it follows by 
\eqref{eq:coeff-val} that $\ord_p \BT_{n,p}(x) = -1$.
On the other side, we can write $n = n'\,p + 1$ with some $n' \geq 1$.
By Lemma~\ref{lem:fs-decomp} and \eqref{eq:fs-eq} we obtain
$\fracsym{n}{p} = \fracsym{n'}{p} + 1 > 1$; showing the claim for this case.

Case $p=2$ and $n \geq 4$ even:
Since $n$ is even, we have for odd $\ell$ with $0 < \ell < n$ that
\[
  \binom{n}{\ell} \equiv \frac{n}{\ell} \binom{n-1}{\ell-1} \equiv 0 \pmod{2}.
\]
Therefore, with $\KK_2 = \set{2,4,\dotsc,n-2}$,
\[
  \ord_p \binom{n}{k} \geq 1 \quad (k \in \KK_2)
  \quad \!\iff\! \quad
  \ord_p \binom{n}{k} \geq 1 \quad (0 < k < n).
\]
With that we can apply Lemma~\ref{lem:fs-val-binom} to show the claim for that 
case.
\end{proof}

\begin{proof}[Proof of Theorem~\ref{thm:main}]
Note that $B_0 = 1$ and $B_1 = -\frac{1}{2}$. 

Cases $n = 1, 2$:
As $\BT_1(x) = x$ and $\BT_2(x) = x^2-x$ by \eqref{eq:bn-poly}, we have for all 
primes $p$ that $\ord_p \BT_n(x) = 0$, while $\fracsym{n}{p} \leq 1$ by 
\eqref{eq:fs-eq}. This shows \eqref{eq:frac-denom} for these cases.

Case $n \geq 3$: 
Since $B_k = 0$ for odd $k \geq 3$, we deduce from \eqref{eq:bn-poly} that
\begin{equation} \label{eq:bt-sum}
  \BT_n(x) = x^n - \frac{n}{2} x^{n-1} +
    \sum_{\substack{k=2\\ 2 \, \mid \, k}}^{n-1} \binom{n}{k} B_k \, x^{n-k}.
\end{equation}
For even $k \geq 2$ we have by \eqref{eq:denom-bn} that
$B_k \in \ZZ_p$ if $p-1 \nmid k$. Thus we infer from~\eqref{eq:bt-sum} that
\begin{equation} \label{eq:bt-val}
  \ord_p \BT_n(x)
  = \min \mleft( 0, \, \ord_p \frac{n}{2}, \, \ord_p \BT_{n,p}(x) \mright).
\end{equation}
Since $\ord_p \BT_{n,p}(x) \geq -1$ by Lemma~\ref{lem:bt-p-val} and
$\ord_p \frac{n}{2} \geq -1$, we get
\begin{equation} \label{eq:bt-val-set}
  \ord_p \BT_n(x) \in \set{-1, 0}.
\end{equation}
If $\ord_p \frac{n}{2} = -1$, then $p=2$ and $n \geq 3$ is odd. In this special
case, we concurrently have $\ord_p \BT_{n,p}(x) = -1$ by Lemma~\ref{lem:bt-p-val}.
Hence, we then conclude from Lemma~\ref{lem:bt-p-val} and \eqref{eq:bt-val} that
\[
  \fracsym{n}{p} > 1 \quad \!\iff\! \quad
  \ord_p \BT_n(x) = -1 \quad \!\iff\! \quad
  p \mid \denom \bigl( \BT_n(x) \bigr).
\]
This shows \eqref{eq:frac-denom} for the case $n \geq 3$.

Finally, since \eqref{eq:bt-val-set} holds in all cases, the denominator of
$\BT_n(x)$ is squarefree. This completes the proof.
\end{proof}


\section{Proof of \texorpdfstring{Theorem~\ref{thm:power}}{Theorem~3} and
\texorpdfstring{Corollary~\ref{cor:power}}{Corollary~1}}

All previous results have been deduced without involving the function $s_p(n)$
directly. However, using an alternative form of Legendre's formula of
$\ord_p n!$, we have besides that
\begin{equation} \label{eq:fs-sp}
  \ord_p n! = \frac{n - s_p(n)}{p-1}
  \quad \text{and} \quad
  \fracsym{n}{p} = \frac{s_p(n)}{p-1}.
\end{equation}

For two multiplicatively independent integers $a,b \geq 2$, a positive integer
$n$ cannot have few nonzero digits in both bases $a$ and $b$ simultaneously.
Steward \cite[Thm.~1, p.~64]{Stewart:1980} gave a lower bound such that
\begin{equation} \label{eq:sp-estim-1}
  s_a(n) + s_b(n) > \frac{\log \log n}{\log \log \log n + c} - 1 \quad (n > 25)
\end{equation}
with an effectively computable constant $c$ depending on $a$ and $b$.
See also Bugeaud \cite[Thm.~6.9, p.~134]{Bugeaud:2012} for a related result.
However, the weaker result of Senge and Straus \cite[Thm.~3]{Senge&Straus:1973},
that for any constant $C$ the number of integers $n$ satisfying
\begin{equation} \label{eq:sp-estim-2}
  s_a(n) + s_b(n) < C
\end{equation}
is finite, would already suffice for our purpose.

\begin{proof}[Proof of Theorem~\ref{thm:power}]
We can write $n = \tilde{n} \, p^r$ with some $r \geq 0$ and $\tilde{n} > 1$,
where $p \nmid \tilde{n}$ by assumption. Taking $a = p$ and $b = n$ in view of
\eqref{eq:sp-estim-1} and \eqref{eq:sp-estim-2}, we have
$s_p(n^k) = s_p(\tilde{n}^k) > 1$ and $s_n(n^k)=1$ for all $k \geq 1$.
We then conclude from \eqref{eq:sp-estim-1}, or similarly from
\eqref{eq:sp-estim-2}, that
\[
  s_p(n^k) \to \infty \quad \text{as} \quad k \to \infty.
\]
By \eqref{eq:fs-sp} the result follows.
\end{proof}

\begin{proof}[Proof of Corollary~\ref{cor:power}]
From Theorem~\ref{thm:power} we infer that for each prime $p \in \PP$ there
exists a constant $m_p$ such that
\[
  \fracsym{n^k}{p} > 1 \quad (k \geq m_p).
\]
Set $M := \max\limits_{p \in \PP} \, m_p$. Then by Theorem~\ref{thm:main} the
result follows.
\end{proof}


\end{document}